\documentclass[10pt]{amsart}
\usepackage{fullpage}
\usepackage[centertags]{amsmath}
\usepackage{amsfonts}
\usepackage{amsthm}
\usepackage{newlfont}
\usepackage{amscd}
\usepackage{amsgen}
\usepackage{amssymb}
\usepackage{easybmat}
\newlength{\defbaselineskip} 

 \theoremstyle{plain} \newtheorem{thm}{Theorem}[section]
 
\newtheorem{prop}[thm]{Proposition}

\theoremstyle{definition}

\newtheorem{rem}[thm]{Remark}

 \numberwithin{equation}{section}
\numberwithin{equation}{section}
\theoremstyle{definition} \newtheorem{ex}{Example}[section]

\title{Mirror symmetry for Pfaffian Calabi-Yau 3-folds via conifold transitions}
\author{Micha{\l}\ Kapustka}
\begin{document}

\begin{abstract}In this note we construct conifold transitions between several
Calabi-Yau threefolds given by Pfaffians 
in weighted projective spaces and Calabi-Yau threefolds appearing as complete
intersections in toric varieties. We use the obtained results to predict mirrors
following ideas of \cite{BCKS, Batsmalltoricdegen}. 
In particular, we consider the family of Calabi--Yau threefolds of degree 25 in
$\mathbb{P}^9$
obtained as a transverse intersection of two Grassmannians in their Pl\"ucker
embeddings.
\end{abstract}

\maketitle
\section{Introduction}
Calabi--Yau threefolds with Picard number one are central objects of
investigation form the point of view of mirror symmetry. The main reason for
this is that for such a manifold the mirror family 
has one dimensional moduli,  and hence can be explicitly studied. 
There are nowadays more than 100 known families of Calabi--Yau threefolds of
Picard number 1. The simplest are complete intersection in toric varieties and
for them the mirror symmetry conjecture has been proven.
Others,  and in fact most of them,  appear as smoothings of hypersurfaces in
some toric Gorenstein terminal Fano fourfolds (see \cite{BatyrevKreuzer}). 
For these a conjectural mirror construction has been developed.

More precisely basing on ideas of \cite{Morrison} in \cite{BCKS,  BCKS2} a
conjectural method for the construction of mirrors for Calabi--Yau threefolds
admitting conifold transitions to complete intersection of toric varieties has
been stated.
Since then toric degenerations of Fano manifolds and degenerations of Calabi-Yau
threefolds
to complete intersections in toric varieties has been widely investigated aiming
at understanding mirror symmetry for new classes of examples.

Recently in \cite{MichUnpr, Kanazawa} new families of non-complete intersection
Calabi-Yau threefolds with Picard number 1 have been explicitly constructed.
They will be denoted $\mathcal{X}_5$,  $\mathcal{X}_7$,  $\mathcal{X}_{10}$ and
$\mathcal{X}_{25}$.
They are described by Pfaffian equations in some weighted projective spaces.
According to the classification of \cite{BatyrevKreuzer} they do not admit any
conifold transition to a hypersurface in a toric variety. 

In \cite{Kanazawa} the families $\mathcal{X}_5$,  $\mathcal{X}_7$, 
$\mathcal{X}_{10}$ have also been studied,  together with the classical family
$\mathcal{X}_{13}$ of Calabi--Yau threefolds of degree 13 in $\mathbb{P}^6$, 
from the point of view of mirror symmetry.
The method used relies on the tropicalization approach introduced in
\cite{Boehm}.
In this way all these examples have been assigned a candidate mirror family and
the period of these families has been computed. The singularities of the
elements of the mirror families proposed are however very complicated. In
particular, it is not clear
whether general elements of the mirror families proposed admit resolutions being
Calabi--Yau threefolds. 

In this note, we study mirror symmetry for all examples $\mathcal{X}_5$,
$\mathcal{X}_7$,  $\mathcal{X}_{10}$,  $\mathcal{X}_{25}$ and $\mathcal{X}_{13}$
using the methods of \cite{BCKS, Batsmalltoricdegen}. 
We start by interpreting the description of these Calabi--Yau threefolds as
complete intersection in some singular Fano varieties related to weighted
Grassmannians. 
By constructing toric degenerations of these ambient varieties, we describe
conifold transitions between the families $\mathcal{X}_5$, $\mathcal{X}_7$, 
$\mathcal{X}_{10}$, $\mathcal{X}_{13}$, $\mathcal{X}_{25}$,  and some Calabi-Yau
complete intersections in
toric varieties. We build on the well known toric degeneration of the
Grassmannians $G(2, 5)$
described in \cite{Sturm} and used in \cite{BCKS}. More precisely, we adapt it to
the case of any variety described by Pfaffians of a skew-symmetric $5\times5$
matrix in weighted projective space. 
We next use the methods of \cite{BCKS, Batsmalltoricdegen} to compute the main
period of the conjectured mirror family.
In this way, we recover the same periods as in \cite{Kanazawa} for the examples
studied there. One of the advantage of taking our approach is that our
constructions involving only conifold transitions leads to candidate mirror 
families consisting of singular Calabi--Yau threefolds which conjecturally (see
\cite{Morrison,  BCKS}) admit only nodes as singularities and hence can be
resolved to smooth Calabi--Yau threefolds. Furthermore the method of constructing mirrors via conifold transitions
has good chances to be proved in a general context (see \cite{RuddatSiebert}).

Moreover, our method works also for the family $\mathcal{X}_{25}$ consisting of
Calabi--Yau threefolds obtained as  complete 
intersections of two Grassmannians in their Plucker embeddings. Our approach to
this case can be extended to work for other Calabi-Yau threefolds appearing as
intersections of two Fano varieties 
each admitting a small toric degeneration.

The family $\mathcal{X}_{25}$ consisting of codimension 6 smooth Calabi--Yau
threefolds of degree 25 in $\mathbb{P}^9$ is especially interesting because its
associated Picard--Fuchs equation appears to be self dual in the sense of
\cite{Rodland, Kanazawa}. 
This phenomenon is related to projective self-duality of the Grassmannian
$G(2,5)$ and merits further investigation. We hope that the results obtained in
this paper will contribute to it. For this reason we provide also an alternative approach valid only for this family.
It is based on a construction of a smooth degeneration of the Calabi-Yau manifolds from $\mathcal{X}_{25}$ to Calabi--Yau manifolds 
obtained as zero loci of a vector bundle on $G(2,5)$. This permits the family $\mathcal{X}_{25}$ to be investigated by methods 
specific to zero loci (forthcoming work of S. Galkin).

Throughout the paper, we rely on computer calculations. We use mainly the Toric
Package from the computer algebra system Magma (see \cite{magma}).
\section{Descriptions of studied families}\label{sec constructions}
In this section, we recall and study descriptions of the families
$\mathcal{X}_5$, $\mathcal{X}_7$,  $\mathcal{X}_{10}$ and $\mathcal{X}_{25}$ of
Calabi-Yau threefolds with Picard number one constructed in
\cite{MichUnpr, Kanazawa}. For each of them we provide a description in five
different ways:
\begin{enumerate}
 \item in terms of Pfaffian varieties associated to decomposable vector bundles
in weighted projective spaces, 
 \item using Pfaffian equations in weighted projective space, 
 \item as a result of a bitransition based on Kustin--Miller unprojection, 
 \item as a complete intersection in some cone over some universal Pfaffian
variety, 
 \item as the result of a conifold transition with a complete intersection in a
smooth toric Fano variety. 
\end{enumerate}
Each of the above description has its own advantages. 
Descriptions (1) and (2) are strictly related and very explicit,  they are used
for the definition of the families. In fact the description of varieties using
Pfaffians is in general one of the simplest
after descriptions as complete intersections and by zero loci of sections of
ample vector bundles. 
As such they can be used for the study of the geometry of varieties involved.
However from the point of view of mirror symmetry the description via Pfaffians
does not help much. 
In particular, it seams very improbable that a version of Quantum Lefschetz
theorem could be developed in this case,  since there are examples of Pfaffian varieties for which the assertion of the standard Lefschetz theorem fail.
However, having explicit equations, one can always try to find explicit
degenerations suitable for different approaches to mirror symmetry.

Description (3) tells us about the place the Calabi--Yau threefolds in question
take in the Web of Calabi--Yau threefolds. It was introduced in \cite{MichUnpr}.
It relates our varieties with very standard Calabi--Yau threefolds by
composition of a conifold transition and a geometric transition 
involving a type II primitive contraction morphism. Such constructions are
conjecturally (see \cite{Morrison}) compatible with mirror symmetry,  so in
principle could lead to the construction of a mirror family for our examples.
However contrary to 
conifold transitions the geometric transition involving a primitive contraction
of type II has not yet found a proper counterpart in this theory.

Description (4) is already more suited for mirror symmetry in general. Since in
some instances the quantum Lefschetz theorem on singular varieties holds one can
in principle reduce the study of mirror symmetry of our Calabi--Yau threefold to
the 
mirror symmetry of the ambient variety in question. In our cases the latter has
not yet been studied. It seams however probable that since the ambient varieties
obtained in our cases are strictly related to weighted Grassmannians one could 
generalize the theory developed for Grassmannians to study the quantum
cohomology ring of these varieties. In this paper, we shall just use the analogy
with the Grassmannian to construct a terminal Toric degeneration of our ambient
space and get description (5).

The last description is the one that we shall use to study mirror symmetry for
our examples. It provides a setting in which the methods of \cite{BCKS}
conjecturally work. 
The constructions for all considered families are very similar. We shall write
in details only the cases $\mathcal{X}_5$. For the remaining we present only the
main results and point out where some differences to the case $\mathcal{X}_5$
occur. 
Moreover, the construction (3) using bitransitions is only performed for the case
$\mathcal{X}_5$ which is the only case which was not treated in \cite{MichUnpr}.
For the remaining cases a detailed analysis of the construction can be found
there.
\subsection{The family $\mathcal{X}_5$}
A Calabi-Yau threefold $X_5\in \mathcal{X}_5$ is naturally embedded in
$\mathbb{P}(1^4, 2^3)$.
It is defined as a Pfaffian variety associated to the vector bundle:
$$\mathcal{E}_5=5 \mathcal{O}_{\mathbb{P}{(1^4, 2^3)}}(1).$$
In other terms it is defined by $4\times4$ Pfaffians of a $5\times5$
antisymmetric matrix  with entries of weighted degree 2.
We shall use the following picture to illustrate the weights in the skew
symmetric matrix defining $X_5$ 

$$\left(\begin{array}{ccccc}
    &2&2&2&2\\
 & &2&2&2\\
 & & &2&2\\
 & & & &2\\
 & & &&\\
\end{array}\right).$$
    
The threefold $X_5$ can also be described as a smoothing of a variety obtained
as the result of a Kustin--Miller unprojection of a quadric surface in a
Calabi-Yau threefold complete intersection of two quartics in
$\mathbf{P}=\mathbb{P}(1^4, 2^2)$. More precisely let $D$ be a quadric surface
embedded in $\mathbb{P}(1^4, 2^2)$ as a complete intersection of type $(2, 2,
2)$ i.e. defined by three polynomials $q_1$,  $q_2$,  $q_3$ of weighted degree
2. 
Let us consider a general variety $Y$ obtained as the complete intersection of
two quartics containing $D$ i.e. $Y$ is defined by two polynomials of the form
$p_1=a_1q_1+a_2q_2+a_3q_3$ and $p_2=b_1q_1+b_2q_2+b_3q_3$, 
where $a_i$, $b_i$ are general polynomials of weighted degree 2. Let us now
consider the space  $\mathbb{P}(1^4, 2^3)$ containing $\mathbf{P}$ with the new
weight 2 variable denoted by $l$.
Then the variety $Z$ defined by the $4\times 4$ Pfaffian of the matrix 
$$\left(\begin{array}{ccccc}
    &l&a_1&a_2&a_3\\
 & &b_1&b_2&b_3\\
 & & &q_3&-q_2\\
 & & & &q_1\\
 & & & &\\
\end{array}\right), $$
is a Gorenstein Calabi--Yau threefold whose projection from the point $p$ with
$l(p)=1$ and the remaining coordinates being zero is $Y$. We can easily prove
that this projection factors through a small resolution of nodes on $Y$ and a
primitive contraction of a quadric surface. We moreover observe that $Y$ has a
smoothing to the family of Calabi--Yau threefolds obtained as intersections of
general quartics in $\mathbf{P}=\mathbb{P}(1^4, 2^2)$ whereas $Y$ is smoothed 
by the family $\mathcal{X}_{5}$

The description of the threefold $X_5$ using Pfaffian equations enables us to
see $X_5$ as a general complete intersection of a variety 
$G_5$ described by equations in $\mathbb{P}(1^4, 2^{10})$ given by $4\times 4$
Pfaffians of a matrix $M$ with weight 2 coordinates 
as entries. The variety $G_5$ can be interpreted as a weighted cone over the
Grassmannian $G(2, 5)$.

Observe that $G_5$ is a normal Gorenstein Fano variety. Let us now consider the
following degeneration of $G_5$. 
Let $\mathcal{G}$ be the family defined in $\mathbb{P}(1^4, 2^{10})\times
\mathbb{C}$ with coordinates $x_1, \dots x_4,  y_1, \dots y_{10}$ by the
Pfaffians of the matrix

$$\left(\begin{array}{ccccc}
    &\lambda y_1&y_2&y_3&y_4\\
 & &y_5&y_6&y_7\\
 & & &y_8&y_9\\
 & & & &\lambda y_{10}\\
 & & & &\\
\end{array}\right).$$

\begin{prop}\label{flat and fan}
The family $\mathcal{G}$ is flat over $\mathbb{C}$. Moreover the fiber
$F_5=\mathcal{G}_{0}$ over $\lambda=0$ is a terminal Gorenstein toric Fano
variety of Picard number one. 
\end{prop}
\begin{proof} We start the proof with the observation that $F_5$ is of expected
codimension in $\mathbb{P}(1^4, 2^{10})$ it is hence a (weighted) Pfaffian
variety.
For flatness of $\mathcal{G}$ we then just need to observe that it is clearly an
algebraic family and that the Hilbert polynomial are computed from the same
Pfaffian sequence (cf. \cite{Kanazawa}). 
To get assertions concerning the fiber $F_5=\mathcal{G}_{0}$ we first need to
prove that it is normal. 
Since it is Gorenstein by the Pfaffian construction,  normality is equivalent to
the computation of the codimension of the singular locus. The latter is done
easily by computer. Since the Pfaffians of the degenerate matrix
provide binomial equations for $F_5$ normality implies that $F_5$ is a toric
variety. The rest of the assertion follows 
from computer calculation using the Toric Package in Magma on the fan of
$F_5=\mathcal{G}_{0}$ determined by the binomial equations. 
More precisely we compute that the polytope associated to $F_5=\mathcal{G}_{0}$
polarized by the restriction of $\mathcal{O}_{\mathbb{P}(1^4, 2^{10})}(1)$ is a
reflexive polytope generated by:
 $$\begin{array}{rrrrrrrrrrrr}
       e_1&=&(-1, & 0, & -1, & 0, & 0, & 0, & 0, & 0, & 0, & -1), \\
  e_2&=& (0, & -1, & 0, & 0, & 0, & 0, & 0, & 0, & 0, & 1), \\
  e_3&=&  (0, & -1, & 1, & 0, & 0, & 0, & 0, & 0, & 0, & 0), \\
  e_4&=&  (0, & 0, & 0, & 0, & 0, & 0, & 0, & 0, & 1, & 0), \\
  e_5&=&  (0, & 0, & 0, & 0, & 0, & 0, & 0, & 1, & 0, & 0), \\
  e_6&=&  (0, & 0, & 0, & 0, & 0, & 0, & 1, & 0, & 0, & 0), \\
  e_7&=&  (0, & 0, & 0, & 0, & 0, & 1, & 0, & 0, & 1, & 0), \\
  e_8&=&  (0, & 0, & 0, & 0, & 0, & 1, & 1, & 0, & 0, & 0), \\
  e_9&=&  (0, & 0, & 0, & 0, & 1, & 0, & 0, & 1, & 0, & 0), \\
  e_{10}&=&  (0, & 0, & 0, & 0, & 1, & 1, & 1, & 0, & 0, & 0), \\
  e_{11}&=&  (0, & 0, & 0, & 1, & 0, & 0, & 0, & 0, & 0, & 0), \\
  e_{12}&=&  (0, & 1, & 0, & -1, & -1, & -1, & -1, & -1, & -1, & 0), \\
  e_{13}&=&  (1, & 0, & 0, & 0, & 0, & 0, & 0, & 0, & 0, & 0)
   \end{array} 
$$
It follows by further calculations with Magma that the toric variety $F_5$ is a
terminal Gorenstein variety with Picard number one with two singular strata of
codimension 3 corresponding to the cones generated by $(e_5, e_8, e_9, e_{10})$
and $(e_4, e_6, e_7, e_8)$. 
\end{proof}
From the proof of Proposition \ref{flat and fan} we obtain that the intersection
of $F_5$ with 7 general hypersurfaces of weighted degree 2 is a nodal
Calabi--Yau threefold $T_5$.
Observe that $T_5$ admits a small resolution $\tilde{T}_5$ to a complete
intersection in a toric variety obtained as a toric resolution $\tilde{F}_5$ of
$F_5$ (with Fan given by a triangulation of our polytope).
This means that $X_5$ and $\tilde{T}_5$ are connected by a conifold transition.

\subsection{The family $\mathcal{X}_7$}
The Calabi-Yau threefold $X_{7}$ is described in $\mathbb{P}(1^5, 2^2)$ by
$4\times4$ Pfaffians of a $5\times5$ antisymmetric matrix with 
entries of degrees as shown below: 
 $$\left(\begin{array}{ccccc}
    &1&1&2&2\\
& &1&2&2\\
& & &2&2\\
& & & &3\\
& & & &\\
\end{array}\right)$$
We can consider $X_7$ as complete intersection of hypersurfaces of degrees
$2$,  $2$,  $2$,  $2$,  $3$ and the variety $G'_7\subset \mathbb{P}(1^5, 2^6,
3)$ defined by the
$4\times 4$ Pfaffians of the 
antisymmetric matrix with entries being coordinates of suitable degree. Now
similarly to the case $G(2, 5)$ we can find a toric degeneration of $G'_7$. The
latter is however not Gorenstein.
To obtain a degeneration which is Gorenstein we consider $G_7$ to be the variety
defined by Pfaffians of 
a generic matrix of the form 
 $$\left(\begin{array}{ccccc}
    &x_1&x_2&y_1&y_2\\
& &x_3&y_3&y_4\\
& & &y_5&y_6\\
& & & &c\\
& & & &\\
\end{array}\right)$$
in $\mathbb{P}(1^5, 2^6)$ where $x_1,x_2,x_3$ are weight one coordinates
$y_1\dots y_6$ weight 2 coordinates and $c$ a general polynomial of weighted
degree 3. 
In this way we obtain a family of varieties $G_7$.
For each of them we have a degeneration, similar to the one described for $G_5$,
to the same variety $F_7$ in $\mathbb{P}(1^5, 2^6)$. This time $F_7$ is a
terminal Gorenstein toric Fano variety polarized by
$\mathcal{O}_{\mathbb{P}(1^5, 2^6)}(1)$ with polytope:
\begin{equation*}
 \begin{array}{lll}
 e_1= (-1,  -1,  0,  -1,  0,  0,  0), &
 e_2=   (0,  0,  -1,  0,  1,  1,  0), &
 e_3=   (0,  0,  -1,  1,  0,  0,  0), \\
 e_4=   (0,  0,  0,  0,  1,  0,  1), &
 e_5=   (0,  0,  0,  0,  1,  1,  1), &
 e_6=   (0,  0,  1,  0,  -1,  -1,  -1), \\
 e_7=   (0,  0,  1,  0,  -1,  0,  -1), &
 e_8=   (0,  1,  0,  0,  0,  0,  0), &
 e_9=   (1,  0,  -1,  0,  0,  0,  0), \\
 e_{10}=   (1,  0,  0,  0,  0,  0,  1) & &  
   \end{array}
\end{equation*}
The variety obtained as the intersection of 4 general Cartier divisors from the
system corresponding to $\mathcal{O}_{\mathbb{P}(1^5, 2^6)}(2)$ in $F_7$ is a
nodal Calabi--Yau threefold. Its resolution $\tilde{T}_7$ is a complete
intersection in a toric 
resolution $\tilde{F}_7$ of $F_7$ and is connected to $X_7$ by a conifold
transition.

\subsection{The family $\mathcal{X}_{10}$}

The Calabi-Yau threefold $X_{10}$ is described in $\mathbb{P}(1^6, 2)$ by
$4\times4$ Pfaffians of a $5\times5$ antisymmetric matrix with 
entries of degrees as shown below: 
 $$\left(\begin{array}{ccccc}
    &1&1&1&1\\
 & &2&2&2\\
 & & &2&2\\
 & & & &2\\
 & & & &\\
\end{array}\right)$$
We consider $X_{10}$ as complete intersection of hypersurfaces of weighted
degrees
$2, 2, 2, 2, 2$ and the variety $F_{10}\subset \mathbb{P}(1^6, 2^6)$ defined by
the
$4\times 4$ Pfaffians of the 
antisymmetric matrix with entries being coordinates of suitable degree. Now
similarly to the case $G(2, 5)$ we find a toric degeneration of $F_{10}$.
It is polarized by the restriction of $\mathcal{O}_{\mathbb{P}(1^6, 2^6)}(1)$
and corresponds to the polytope.
$$\begin{array}[c]{cccccccccc}
e_1=( &-1, & -1, & -1, & 0, & -1, & 0, & 0, & 0&), \\
e_2=( &0, & 0, & 0, & -1, & 0, & 1, & 0, & 0&), \\
e_3=( &0, & 0, & 0, & -1, & 1, & 0, & 0, & 0&), \\
e_4=( &0, & 0, & 0, & 0, & 0, & 0, & 0, & 1&),  \\
e_5=( &0, & 0, & 0, & 0, & 0, & 1, & 1, & 0&), \\
e_6=( &0, & 0, &  0, & 1, & 0, & -1, & -1, & -1&), \\
e_7=( &0, & 0, &  1, & 0, & 0, & 0, & 0, & 0&), \\
e_8=( &0, & 1, &  0, & 0, & 0, & 0, & 0, & 0&), \\
e_9=( &1, & 0, &  0, & 0, & 0, & 0, & 1, & 1&), \\
e_{10}=( &1, & 1,  & 0, & -1, & 0, & 0, & 0, & 0&), \\
e_{11}=( &1, & 1,  & 0, & 0, & 0, & 0, & 1, & 0&).  \\
\end{array}
$$
The variety obtained as the intersection of 5 general Cartier divisors from the
system corresponding to $\mathcal{O}_{\mathbb{P}(1^6, 2^6)}(2)$ in $F_{10}$ is a
nodal Calabi--Yau threefold. 
Its resolution $\tilde{T}_{10}$ is a complete intersection in a toric resolution
$\tilde{F}_{10}$ of $F_{10}$ and is connected to $X_{10}$ by a conifold
transition.

\subsection{The family $\mathcal{X}_{13}$} By the same method we can also treat
the Tonoli examples of degree 13.
A Calabi-Yau threefold $X_{13}$ from the family $\mathcal{X}_{13}$ is described
in $\mathbb{P}^6$ by $4\times4$
Pfaffians of a $5\times5$ antisymmetric matrix with 
entries of degrees as shown below: 
 $$\left(\begin{array}{ccccc}
    &2&2&2&2\\
 & &1&1&1\\
 & & &1&1\\
 & & & &1\\
 & & & &\\
\end{array}\right)$$
We consider $X_{13}$ as a complete intersection of 4 hypersurfaces of degrees
$2$  in $\mathbb{P}(1^7, 2^4)$ and the subvariety 
 $G_{13}\subset \mathbb{P}(1^7, 2^4)$ defined by $4\times4$
Pfaffians of a $5\times5$ antisymmetric matrix with 
entries being coordinates of suitable degree. Again we find a toric degeneration
of $G_{13}$ by the same method and in consequence a conifold transition from
$X_{13}$ to a Calabi--Yau threefold obtained as a complete intersection in a smooth 
toric variety.

\section{Mirror symmetry via toric degenerations}   In this section, we recall a
construction based on small toric degenerations
which is used to conjecturally predict the principal period of the mirror
family of a given Calabi-Yau threefold with Picard number 1. It is the original
method of Batyrev (\cite{Batsmalltoricdegen})
describing the principal period as a specialization 
of the hyper-geometric series associated with the toric resolution of the
degenerate Fano manifold.

In fact, using this method one obtains an explicit candidate for the mirror
family of $X_i$. More precisely,  the mirror family $T^*_i$ of $T_i$ is computed
explicitly  in terms of the construction of \cite{BatyrevBorisov}. 
To get an explicit description of the candidate mirror family for $X_i$ we use a
specialization of the family $T^*_i$ analogous to \cite[Conjecture
6.1.2]{BCKS}. 
It is conjectured that elements of this specialized family admit only nodes as
singularities and their small resolutions are mirrors to $X_i$.
Since we are unable to prove the conjecture on singularities in this case we
omit the details of the construction of the mirror family here and we
concentrate on the computation of its main period.

The method presented in \cite{Batsmalltoricdegen} is the following.
Let $X$ be a Calabi-Yau threefold appearing as a complete intersection of
Cartier divisors $D_1, \dots D_n$ in a Fano variety $G$ of dimension $n+3$.
Assume that $G$ admits a small degeneration to a toric variety $F$ i.e. a flat
degeneration such that $F$ is a terminal Gorenstein Fano variety and such that 
there is a canonical isomorphism between $\operatorname{Pic}(G)$
and $\operatorname{Pic}(F)$,  denote $\tilde{D}_1,  \dots,  \tilde{D}_n$ the
Cartier divisors on $F$ corresponding to $D_1, \dots D_n$ via this 
isomorphism.
Let $B=\{e_1, \dots,  e_k\}\subset N$ be the generators of one-dimensional cones
of the fan $\Sigma$ of $F$ in 
the dual lattice $N=\mathbb{Z}^n$. It is well known that the vectors
$\{e_1, \dots, e_k\}$ determine a set $\{E_1, \dots,  E_k\}$ of generators of 
the divisor class group. Assume that we have a subdivision of $B$ into $n$
disjoint sets $J_1, \dots, J_n$ such that $J_i$ corresponds to 
the Cartier Divisor $\tilde{D}_i$ for each $i\in\{1, \dots, n\}$ (i.e.
$D_i=\sum_{j\in J_i} E_j$). 
Let $L(B):=\{(l_1, \dots, l_k)\in \mathbb{Z}^k: \sum_{j=1}^k l_j e_j=0,  \quad
l_i\geq 0\}$. We then have a pairing between $L(B)$ and 
$\operatorname{Cl}(F)$ given by $<(l_1, \dots, l_k), E_j>=l_j$.
Finally we call $A(\Sigma)$ the set of vectors in $\mathbb{C}^{k}$ admissible
for the fan $\Sigma$ of i.e. vectors $(a_1, \dots, a_k)\in \mathbb{C}^k$ such
that there 
exists a function $\varphi$ on $\mathbb{C}^{n+3}$ linear restricted to each cone
of 
$\Sigma$ and such that $\varphi(e_i)=\log |a_i|$ for each $i\in\{1, \dots, k\}$.
Under this notation,  the main period of the mirror family to $X$ is conjectured
to be given by the formula:
\begin{equation} \label{rownanie} 
\phi_0(z)=\sum_{l\in L(B)} \frac{\prod_{i=1}^{n}(\sum_{j\in J_i} l_j)!}{l_1!
\dots l_k!} \prod_{j=1}^k z_j^{l_j}, 
\end{equation}
where $z\in A(\Sigma)$.
\begin{rem}
 Observe that in \cite{Batsmalltoricdegen}  the variety $G$ is assumed to be a
smooth Fano variety. However the method is conjectured to work for any conifold
transition between $X$ and a Calabi--Yau threefold
 obtained as a complete intersection in a terminal Gorenstein toric variety.
\end{rem}

\begin{ex}
The Calabi-Yau threefold $X_5$ admits a degeneration to a complete intersection
in a Gorenstein terminal toric Fano variety $F_5$ of dimension 10.
The Picard number of $F_5$ is 1 and the generator of the Picard group is very
ample.
The following decomposition of the set $B_5=\{e_1\dots e_{13}\}$ of rays of the
Fan $\Sigma_5$ of $F_5$ corresponds to 7 sections by elements of the system :
$J_1=\{e_1, e_2\}$,  $J_2=\{e_3, e_{13}\}$,  $J_3=\{e_4, e_6\}$,  $J_4=\{e_5,
e_9\}$,  $J_5=\{e_{7}, e_{8}, e_{10}\}$,  $J_6=\{e_{11}\}$,  $J_7=\{e_{12}\}$.
We compute also the cone $L(B_5)$ and find out that it is generated over
$\mathbb{Z}_{\geq 0}$ by vectors:
\begin{equation*}
 \begin{array}{cc}
 f_1=(1, 1, 1, 2, 2, 0, 0, 0, 0, 2, 2, 2, 1), &
 f_2=(1, 1, 1, 2, 0, 0, 0, 2, 2, 0, 2, 2, 1), \\
 f_3=(1, 1, 1, 0, 0, 2, 2, 0, 2, 0, 2, 2, 1), &
 f_4=(1, 1, 1, 2, 1, 0, 0, 1, 1, 1, 2, 2, 1), \\
 f_5=(1, 1, 1, 1, 1, 1, 1, 0, 1, 1, 2, 2, 1), &
 f_6=(1, 1, 1, 1, 0, 1, 1, 1, 2, 0, 2, 2, 1).
 \end{array}
\end{equation*}
It is hence a simplicial cone of dimension 3 spanned over a triangle with sides
having 3 points belonging to the lattice $\mathbb{Z}^k$ (2 vertices and the
midpoint).
We observe that the monomials corresponding to the six generators of $L(B_5)$
are equal in $A(\Sigma_5)=\{(z_1, \dots, z_{13})\in \mathbb{C}^{11} | z_5
z_{10}=z_8 z_9, \quad z_4 z_8=z_6z_7\}$. 
We can hence set a new coordinate $t=z^{f_i}$ which is independent on $i\in
\{1\dots6\}$ on $A(\Sigma_5)$. To make explicit the summation over $L(B_5)$ we
observe that every element of $P\in L(B_5)$
has a unique presentation as a sum $P=kf_1+lf_2+mf_3+nf_4+of_5+pf_6$ with $k, l,
m, n, o, p\in \mathbb{Z}_{\geq 0}$ and 
$|n|+|o|+|p|\leq 1$.
It follows that the conjectured formula for the main period of the mirror of
$\mathcal{X}_5$ is:
\begin{equation}
\phi_0(t)= \sum_{s=0}^{\infty} {2s \choose s}^2
\left(\sum_{\substack{k+l+m+n+o+p=s, \\  k, l, m, n, o, p\geq 0\\
|n|+|o|+|p|\in\{0, 1\}
}}  {2s \choose 2k+n+o}{2s \choose 2m+o+p}
{2s \choose 2m+o+p, 2l+n+p,  n+o+2k}\right) t^s,
\end{equation}
where, in the above and further in the paper, we use the notation ${a \choose b, c,  a-b-c}=\frac{a!}{b! c! (a-b-c)!}$.
We check that the corresponding Picard-Fuchs equation is the no. 302 from
\cite{vEvS}. It makes our result agree with the computation in \cite{Kanazawa}.
\end{ex}

\begin{ex}
The Calabi-Yau threefold $X_7$ has a degeneration to a nodal threefold appearing
as a complete intersection in a Gorenstein toric Fano variety $F_7$
of dimension 7.
The decomposition of the set of $B_7$ of rays of the fan $\Sigma_7$ is
$J_1=\{e_1, e_3\}$,  $J_2=\{e_4, e_6\}$,  $J_3=\{e_2, e_8, e_9\}$,  $J_4=\{e_5,
e_7, e_{10}\}$.
The cone $L(B_7)$ is generated by vectors:
\begin{equation*}
 \begin{split}
 f_1=(1, 0, 1, 0, 2, 2, 0, 1, 1, 0), 
 f_2=(1, 1, 1, 0, 1, 2, 0, 1, 0, 1), \\
 f_3=(1, 0, 1, 1, 1, 1, 1, 1, 1, 0), 
 f_4=(1, 1, 1, 1, 0, 1, 1, 1, 0, 1).
 \end{split}
\end{equation*}
\end{ex}
It is spanned over a parallelogram. We have 
$A(\Sigma_7)=\{(z_1, \dots, z_{13})\in \mathbb{C}^{11} | z_2 z_{10}=z_5 z_9,
\quad z_4 z_7=z_5z_6\}$ and the monomials corresponding to vectors $f_i$
generating $L(B_7)$ are equal on $A(\Sigma_7)$. Finally to make explicit the
summation over $L(B_7)$
we observe that if we denote the above  generators of $L(B_7)$ by $f_1\dots f_4$
then every point of $P\in L(B_7)$ has a unique presentation
$P=kf_1+lf_2+mf_3+nf_4$ with $k n=0$.
We get the formula for the main period 
\begin{equation}
\phi_0(t)= \sum_{s=0}^{\infty} {2s \choose s}\left(\sum_{\substack{k+l+m+n=s, \\
 k, l, m, n\geq 0\\ kn=0}}  {2s \choose m+n}{2s \choose s, k+m, l+n}
{2s \choose 2k+l+m, m+n,  l+n}\right) t^s 
\end{equation}
In this way we recover again the same result as \cite{Kanazawa} getting the
Picard -Fuchs equation to be no. 109 from \cite{vEvS}.
\begin{ex}
Consider our Calabi-Yau threefold $X_{10}$. As described in Section \ref{sec
constructions} it admits a degeneration to 
a nodal threefold appearing as a complete intersection in a Gorenstein terminal
toric Fano $F_{10}$ variety of dimension 8.
The decomposition of the set of rays $B_{10}$ of the Fan $\Sigma_{10}$ of
$F_{10}$ is
$J_1=\{e_1, e_3\}$,  $J_2=\{e_2, e_5\}$,  $J_3=\{e_4, e_8\}$,  $J_4=\{e_7, e_9,
e_{10}, e_{11}\}$,  $J_5=\{e_6\}$.
The cone $L(B_{10})$ is generated by vectors:
$$
(1, 1, 1, 2, 1, 2, 1, 0, 0, 0, 1), 
(1, 0, 1, 2, 2, 2, 1, 0, 0, 1, 0), 
(1, 1, 1, 1, 1, 2, 1, 1, 1, 0, 0).$$

It is hence a simplicial cone.
We also compute that $A(\Sigma_{10})=\{(z_1, \dots, z_{11})\in \mathbb{C}^{11} |
z_4 z_{11}=z_8 z_9, \quad z_2 z_{11}=z_5 z_{10}\}$.
We observe that the monomials corresponding to the three generators of
$L(B_{10})$ are equal in $A(\Sigma_{10})$. We can hence set a new coordinate 
$t:=z^{f_i}$ for $z\in A(\Sigma_{10})$ independent of $i\in\{1, 2, 3\}$.
Since every point of $P\in L(B_{10})$ has a unique presentation as
$P=kf_1+lf_2+mf_3$,  with $k, l, m\in \mathbb{Z}_{\geq 0}$. 
The main period is then given by the formula:
\begin{equation}
\phi_0(t)= \sum_{s=0}^{\infty} {2s \choose s}\left(\sum_{\substack{k+l+m=s, \\ 
k, l, m\geq 0}}  {2s \choose k+m}{2s \choose m}
{2s \choose m}{2s-m \choose k} {2s -k-m \choose l}\right) t^s 
\end{equation}
This corresponds to the Picard-Fuchs equation no. 263 from \cite{vEvS} as stated
in \cite{Kanazawa}.
\end{ex}
\begin{rem}
A similar computation holds for $X_{13}$ recovering the Picard-Fuchs equation no
99 from \cite{vEvS}.
\end{rem}
\section{The family $\mathcal{X}_{25}$}
We consider the family of Calabi--Yau threefolds of degree 25 in a separate
section because its description is slightly different from the description of
earlier studied varieties.
It involves two sets of Pfaffian equations. More precisely 
$\mathcal{X}_{25}$ is the family of Calabi-Yau threefold of degree 25 obtained
as transversal intersections of
two Grassmannians $G(2, 5)$ embedded by Pl\"ucker embeddings in $\mathbb{P}^9$.
Its equations are given by $4\times4$ Pfaffians of two generic $5\times 5$
matrices of linear forms.
We would like to deform both Grassmannians simultaneously and obtain a toric
variety as the result of their intersection. This might be impossible to do.
However,  we can set up a picture in which 
both deformations do not interfere with each other and the result is really a
toric variety. For this we consider $X_{25}$ as a complete intersection of 10
hyperplanes in $\mathbb{P}^{19}$ 
with the subvariety $G_{25}\subset \mathbb{P}^{19}$ defined by $4\times 4$
Pfaffians of two 
$5\times 5$ skew symmetric matrices with entries being two disjoint sets of
coordinates. This means that $G_{25}$ is the intersection of two cones $C_1$ and
$C_2$ over Grassmannians $G(2, 5)$ centered in two disjoint $\mathbb{P}^9\subset
\mathbb{P}^{19}$.   
We find a toric degeneration of $G_{25}$,  by degenerating each of the cones
$C_1$ and $C_2$ by means of the standard degeneration of varieties given by
Pfaffians of a $5\times 5$ matrix introduced above.
More precisely let $\mathcal{C}_1, \mathcal{C}_2 \subset \mathbb{P}^{19}\times
\mathbb{C}$ be the families obtained by multiplying the corner entries of the
matrices by $\lambda\in \mathbb{C}$.
We then have the following:
\begin{prop}\label{flat deg X25}
The intersection $\mathcal{C}_1\cap \mathcal{C}_2 \subset \mathbb{P}^{19}\times
\mathbb{C}$ is flat over $\mathbb{C}$. Moreover its fiber over zero denoted by
$F_{25}$ is a Gorenstein terminal toric Fano variety $F_{25}$ of Picard number
one.
\end{prop}
\begin{proof}
It is clear that $\mathcal{C}_1\cap \mathcal{C}_2$ is an algebraic family. Its
fiber $F_{25}$ over 0 is the intersection of two cones $\hat{C}_1$ and
$\hat{C}_2$ constructed similarly to $C_1$ and $C_2$ but over small
degenerations of the Grassmannians $G(2, 5)$.
Since $\hat{C}_1$ and $\hat{C}_2$ are both Gorenstein as Pfaffian varieties it
follows that $F_{25}$ is a Gorenstein variety of codimension 6. Its Hilbert
polynomial is just the product of Hilbert polynomials of $\hat{C}_1$ and
$\hat{C}_2$ which are equal to the
Hilbert polynomials of $C_1$ and $C_2$. It follows that the family
$\mathcal{C}_1\cap \mathcal{C}_2$ is flat. We next prove that $F_{25}$ is
normal. For this it is enough to compute the codimension of its singular locus.
For this observe that by construction,  
since the equation of $\hat{C}_1$ and $\hat{C}_2$ involve disjoint sets of
coordinates,  a point on $F_{25}$ is singular if and only if it is a singular
point of $\hat{C}_1$ or $\hat{C}_2$. Since the centers of the cones are of high
codimension this implies 
that the singular locus of $F_{25}$ is of codimension 3.  
Hence $F_{25}$ is normal. We moreover have a set of binomial equations defining
$F_{25}$. It follows that $F_{25}$ is a Gorenstein toric variety.  
The polytope of $F_{25}$ is generated by the set of rays $B_{25}$ consisting of
the following:
\begin{equation*} 
\begin{array}{c}
e_1=(-1,  -1,  -1,  -1,  -1,  -1,  -1,  -1,  -1,  -1,  -1,  -1,  -1), \\
\begin{array}{cc}
 e_2=      (0,  0,  0,  0,  0,  0,  0,  0,  0,  0,  0,  0,  1), &
 e_3=      (0,  0,  0,  0,  0,  0,  0,  0,  0,  0,  0,  1,  0), \\
e_4=       (0,  0,  0,  0,  0,  0,  0,  0,  0,  0,  1,  0,  0), &
 e_5=      (0,  0,  0,  0,  0,  0,  0,  0,  0,  1,  0,  0,  1), \\
 e_6=      (0,  0,  0,  0,  0,  0,  0,  0,  0,  1,  1,  0,  0), &
 e_7=      (0,  0,  0,  0,  0,  0,  0,  0,  1,  0,  0,  1,  0), \\
 e_8=      (0,  0,  0,  0,  0,  0,  0,  0,  1,  1,  1,  0,  0), &
 e_9=      (0,  0,  0,  0,  0,  0,  0,  1,  0,  0,  0,  0,  0), \\
 e_{10}=   (0,  0,  0,  0,  0,  0,  1,  0,  0,  0,  0,  0,  0), &
 e_{11}=   (0,  0,  0,  0,  0,  1,  0,  0,  0,  0,  0,  0,  0), \\
 e_{12}=   (0,  0,  0,  0,  1,  0,  0,  0,  0,  0,  0,  0,  0), &
 e_{13}=   (0,  0,  0,  1,  0,  0,  0,  0,  0,  0,  0,  0,  0), \\
 e_{14}=   (0,  0,  1,  0,  0,  1,  0,  0,  0,  0,  0,  0,  0), &
 e_{15}=   (0,  0,  1,  1,  0,  0,  0,  0,  0,  0,  0,  0,  0), \\
 e_{16}=   (0,  1,  0,  0,  1,  0,  0,  0,  0,  0,  0,  0,  0), &
 e_{17}=   (0,  1,  1,  1,  0,  0,  0,  0,  0,  0,  0,  0,  0), \\
 e_{18}=   (1,  0,  0,  0,  0,  0,  0,  0,  0,  0,  0,  0,  0).&  \\
 \end{array}
\end{array}
\end{equation*}
It follows that $F_{25}$ is terminal and of Picard number one.
\end{proof}
\begin{rem}
Observe that this set of rays can be thought of as two sets of rays each
describing the standard degeneration $P(2,5)$ of the Grassmannian $G(2, 5)$. 
\end{rem}

From the proof of Proposition \ref{flat deg X25} we deduce that $F_{25}$ has 4
codimension 3 singular toric strata obtained by the intersection given by the 2
codimension 3 toric strata in each $\hat{C}_i$ for $i=1, 2$. 
The variety obtained as the intersection of 10 general Cartier divisors from the
system corresponding to $\mathcal{O}_{\mathbb{P}^{19}}(1)$ in $F_{25}$ is hence
a nodal Calabi--Yau threefold. It admits a small
resolution $\tilde{T}_{25}$ being a complete intersection in a toric resolution
$\tilde{F}_{25}$ of $F_{25}$. It follows that $\tilde{T}_{25}$ is connected to
$X_{25}$ by a conifold transition.

\begin{rem}
 The fact that $T_{25}$ is nodal follows directly form the part of the proof of
\ref{flat deg X25} describing the singularities in codimension 3 from which we
can easily deduce the local type in a general point of each of these
singularities.
\end{rem}

The Calabi-Yau threefold $X_{25}$ thus has a degeneration to a nodal threefold
appearing as a complete intersection of ten hyperplane sections in a Gorenstein
terminal toric Fano variety $F_{25}$
of dimension 13. The decomposition of the set of rays $B_{25}$ of the fan
$\Sigma_{25}$ of $F_{25}$ is
$J_1=\{e_1\}$, $J_2=\{e_9\}$, $J_3=\{e_{10}\}$, $J_4=\{e_{18}\}$,  $J_5=\{e_2,
e_4\}$,  $J_6=\{e_3, e_7\}$,  $J_7=\{e_{11}, e_{13}\}$,  $J_8=\{e_{12},
e_{16}\}$,  
$J_9=\{e_{5}, e_{6}, e_{8}\}$,  $J_{10}=\{e_{14}, e_{15}, e_{17}\}$.
The Cone $L(B_{25})$ is generated by:
\begin{equation*}
 \begin{array}{cc}
 f_1=(1, 0, 0, 1, 1, 0, 1, 0, 1, 1, 0, 0, 1, 1, 0, 1, 0, 1), &
 f_2=(1, 1, 0, 0, 0, 1, 1, 0, 1, 1, 0, 0, 1, 1, 0, 1, 0, 1), \\
 f_3=(1, 1, 1, 0, 0, 0, 0, 1, 1, 1, 0, 0, 1, 1, 0, 1, 0, 1), &
 f_4=(1, 0, 0, 1, 1, 0, 1, 0, 1, 1, 1, 0, 0, 0, 1, 1, 0, 1), \\
 f_5=(1, 1, 0, 0, 0, 1, 1, 0, 1, 1, 1, 0, 0, 0, 1, 1, 0, 1), &
 f_6=(1, 1, 1, 0, 0, 0, 0, 1, 1, 1, 1, 0, 0, 0, 1, 1, 0, 1), \\
 f_7=(1, 0, 0, 1, 1, 0, 1, 0, 1, 1, 1, 1, 0, 0, 0, 0, 1, 1), &
 f_8=(1, 1, 0, 0, 0, 1, 1, 0, 1, 1, 1, 1, 0, 0, 0, 0, 1, 1), \\
 f_9=(1, 1, 1, 0, 0, 0, 0, 1, 1, 1, 1, 1, 0, 0, 0, 0, 1, 1)
 \end{array}
\end{equation*}
Observe that if we denote the following vectors in $\mathbb{Z}^9$ 
 $$(1, 0, 0, 1, 1, 0, 1, 0, 1), (1, 1, 0, 0, 0, 1, 1, 0, 1), (1, 1, 1, 0, 0, 0,
0, 1, 1), $$
 by $g_1$,  $g_2$,  $g_3$ respectively. Then any point of $P\in
L(B_{25})\subset\mathbb{Z}^{18}=\mathbb{Z}^9\times \mathbb{Z}^9$ can be written
in a unique way as $P=(kg_1+lg_2+mg_3, ng_1+og_2+ge_3)$ with $k+l+m=n+o+p$.
 We moreover have  $A(\Sigma_{25})=\{(z_1, \dots, z_{18})\in \mathbb{C}^{18} |
z_2 z_{6}=z_4 z_5, \quad z_3 z_{8}=z_6 z_{7}, \quad,  z_{11} z_{15}=z_{13}
z_{14}, \quad z_{12} z_{17}=z_{15} z_{16}\}$.
 In this way we get the following formula for the period of the mirror of
$X_{25}$.
 
$$  \phi_0(t)= \sum_{s=0}^{\infty} \left(\sum_{\substack{k+l+m=s \\n+o+p=s\\ k,
l, m, n, o, p\geq 0}}  {s \choose k}{s \choose m}{s \choose k, l, m}{s\choose
n}{s\choose p}{s\choose n, o, p}
\right) t^s$$
The latter implies 
$$\phi_0(t)=\sum_{s=0}^{\infty} \left(\sum_{\substack{k+l+m=s\\  k, l, m\geq 0}}
 {s \choose k}{s \choose m}{s \choose k, l, m}\right)^2 t^s.
$$
 
The corresponding Picard--Fuchs equation is no 101 in \cite{vEvS}. Indeed,  the
invariants of $X_{25}$ fit with the predicted (in \cite{vEvS}) invariants of a
hypothetical Calabi--Yau threefold of Picard number one with this equation 
describing the period of its mirror.  
\begin{rem} Observe that in the above the vectors $g_i$ are the generators of
the cone $L(B_{P(2, 5)})$ computed for the standard small toric degeneration
$P(2, 5)$ of the Grassmannian $G(2, 5)$. 
\end{rem}

\begin{rem}
The approach to the case $\mathcal{X}_{25}$ seams to work for the computation of
the main period of the mirror family of any Calabi-Yau threefold (or of a Landau
Ginzburg model of any Fano manifold) obtained as a transversal intersection of 
two Fano varieties admitting Gorenstein terminal toric degenerations. More
precisely let $X$ and $Y$ be two Fano manifolds in $\mathbb{P}^{N}$ intersecting
transversely in $Z$. Assume that $X$ and $Y$ admit small toric degenerations 
$T_{X}\subset \mathbb{P}^N$ and $T_{Y}\subset \mathbb{P}^N$. Let $C_{X}$ and
$C_{Y}$ be cones in $\mathbb{P}^{2N+1}$ over $X$ and $Y$ respectively with
vertices being disjoint $\mathbb{P}^N$'s in $\mathbb{P}^{2N+1}$. Observe that
$Z$ is
a complete intersection of $N$ hyperplane sections of $C_{X}\cap C_{Y}$. Then by
analogous proof to Proposition \ref{flat deg X25} we get the intersection of the
cones $C_{T_X} \cap C_{T_Y}$ (defined in the same way as  $C_{X}$ and $C_{Y}$
but 
spanned over $T_X$ and $T_Y$) and $N+1$ hyperplanes is a conifold degeneration
$T_Z$ of $Z$. Moreover $C_{T_X} \cap C_{T_Y}$ is a toric variety with fan
constructed in terms of the fans of $T_X$ and $T_Y$. It means, in particular,
that the set of 
rays $B_{T_Z}$ can be decomposed into two parts one corresponding to rays
$B_{T_X}$ and the other to rays of $B_{T_Y}$. In this way the Cone $L(B_{T_Z})$
will 
be interpreted as the intersection of the products of cones $L(B_{T_X})\times
L(B_{T_Y})$ with a hyperplane. Since the decomposition $J$ of the set of rays
can be done accordingly to the decomposition onto two parts,  at the end we
obtain the 
coefficients of the main period of the mirror of $Z$ to be products of
coefficients of two series each obtained by the naive application of the Batyrev
formula \ref{rownanie}
to suitable Calabi--Yau (not necessarily 3 dimensional) complete intersections
in $X$ and $Y$.
\end{rem}

\begin{rem}
It is interesting to observe that the above is consistent with another method of
computation of the main period of the mirror.
The latter method is based on Przyjalkowski constructions of weak
Landau-Ginzburg models for smoothings of Gorenstein terminal toric Fano
varieties and the quantum Lefschetz formula in such manifolds.
For more details see \cite{PrzWeakLGFano1}. More precisely in this method we
consider the polytope of $F_i$ and associate to it a Laurent polynomial
$\mathcal{P}$. To the latter we associate its constant term series
$I_{\mathcal{P}}$ and use 
the formula \cite[Cor 4.2.2.]{PrzWeakLGFano68} for the computation of the main
period of $X_i$. In each case we obtain the same result as above even if our
variety $F_i$ has no smoothing. In fact the lack of smoothing may be solved by
means of 
\cite{PrzWeakLGFanosingtor}. More precisely $F_i$ has a partial smoothing $G_i$
whose resolution does not affect $X_i$ we can hence for our purposes work on
$G_i$ as if it was smooth.   
 
\end{rem}
Note also that there is one more method that should lead to a construction of a mirror family for Calabi--Yau manifolds from $\mathcal{X}_{25}$.
It is based on the following. 
\begin{prop} There exists a smooth family of Calabi--Yau manifolds over a disc $\Delta$ whose general element is an element
 of the family $\mathcal{X}_{25}$ whereas the central element (over $0$)  is a zero locus of a  general section of the vector bundle $\mathcal{Q}^*(2)$ over $G(2,5)$. Here $\mathcal{Q}$ is the universal quotient bundle on $G(2,5)$.
\end{prop}
\begin{proof} Observe that element of $\mathcal{X}_{25}$ are pfaffian varieties on $G(2,5)$ associated to the bundle  $E=5\mathcal{O}_{G(2,5)}$ i.e. they are codimension 3 submanifolds described as degeneracy loci of skew symmetric maps $E^*(-1)\to E$. 
Consider now the universal exact sequence
$$0\rightarrow \mathcal{U}\rightarrow 5\mathcal{O}_{G(2,5)}\rightarrow \mathcal{Q}\rightarrow 0.$$
By  \cite[prop. 7.2]{CYP6} (we need to change the base from $\mathbb{P}^6$ to $G(2,5)$, the proof remaining unchanged) there exists a flat family whose general element is a pfaffian variety associated to the bundle $E$ as above whereas the special element is a general pfaffian variety associated to the bundle $\mathcal{U}\oplus \mathcal{Q}$. Now since $\wedge^2 \mathcal{U}(1)=\mathcal{O}_{G(2,5)}$ the degeneracy locus of a general skew symmetric map $(\mathcal{U}\oplus \mathcal{Q})^*(-1)\to \mathcal{U}\oplus \mathcal{Q}$ is equal to a pfaffian locus of some skew-symmetric map $\mathcal{Q}^*(-1)\to \mathcal{Q}$. Finally since $\mathcal{Q}$ is of rank 3 the latter is  the zero locus of the corresponding section $\wedge^2 \mathcal{Q}(1)=\mathcal{Q}^*(2)$. The zero locus of such a section is indeed a smooth Calabi-Yau threefold since $\mathcal{Q}^*(2)$ is globally generated and $c_1(\mathcal{Q}^*(2))=3h$ with $h$ the hyperplane class in the Pl\"ucker embedding.
\end{proof}
\begin{rem} Note that the varieties from the family $\mathcal{X}_{25}$ are themselves not zero loci of sections of the vector bundle $\mathcal{Q}^*(2)$. In fact, by dimension count the dimension of the family $\mathcal{X}_{25}$ is bigger than the dimension of the space of sections of $\mathcal{Q}^*(2)$. More precisely, the family of special Calabi--Yau threefolds being zero loci form a divisor in the deformation space of $\mathcal{X}_{25}$ .
\end{rem}

\subsection*{Acknowledgments}
 We would like to thank J. Buczy\'nski,   A. Kanazawa,  G. Kapustka,  Ch.
Okonek,  V. Przyjalkowski,  D. van Straten for discussions and answering
questions related to the subject of this paper.
 The project was supported MNSiW,  N N201 414539 and by the Forschungskredit of
the University of Zurich.

\bigskip

\bibliographystyle{alpha}
\bibliography{biblio}
\vskip 2cm

\begin{minipage}{15cm}
Department of Mathematics and Informatics,\\ Jagiellonian
University, {\L}ojasiewicza 6, 30-348 Krak\'{o}w, Poland.\\
\end{minipage}

\begin{minipage}{15cm}
Institut f\"ur Mathematik\\
Mathematisch-naturwissenschaftliche Fakult\"at\\
Universit\"at Z\"urich, Winterthurerstrasse 190, CH-8057 Z\"urich\\
\end{minipage}

\begin{minipage}{15cm}
Institutt for matematikk og naturvitenskap\\
Det teknisk- naturvitenskapelige fakultet\\
Universitetet i Stavanger, 4036 Stavanger, Norway\\
\end{minipage}

\begin{minipage}{15cm}
\emph{E-mail address:} michal.kapustka@uj.edu.pl
\end{minipage}

\end{document}